\documentclass[10pt]{amsart}
\usepackage{amscd, amsfonts, amsmath, amssymb, amsthm, fixmath, mathrsfs}
\usepackage[all]{xy} 
\makeindex

\theoremstyle{definition} 
\newtheorem{Unity}{Unity}[section] 
\newtheorem{Definition}[Unity]{Definition} 

\theoremstyle{plain} 
\newtheorem{Theorem}[Unity]{Theorem}
\newtheorem*{Theorem*}{Theorem}
\newtheorem{Proposition}[Unity]{Proposition}

\newtheorem{Lemma}[Unity]{Lemma}

\theoremstyle{remark} 

\newcommand{\GL}{\mathrm{GL}}

\newcommand{\E}{{\mathscr E}}
\newcommand{\F}{{\mathscr F}}
\newcommand{\Ker}{\mathrm{Ker}}

\newcommand{\Omg}{\mathrm{\Omega}}

\newcommand{\rk}{\mathrm{rk}}
\newcommand{\dg}{\mathrm{deg}}

\begin{document}

\title{Instability of Truncated Symmetric Powers of sheaves}
\author{Lingguang Li}
\author{Fei Yu}
\address{School of Mathematical Sciences, Fudan University, Shanghai, P. R. China}
\email{LG.Lee@math.ac.cn}
\address{Universit{\"a}t Mainz, Fachbereich 17, Mathemati, 55099 Mainz, Germany}
\email{yuf@uni-mainz.de}
\maketitle

\begin{abstract} Let $X$ be a smooth projective variety of dimension $n$ over an
algebraically closed field $k$ of characteristic $p>0$. Let $F_X:X\rightarrow X$ be the absolute Frobenius morphism, and $\E$ a torsion free sheaf on $X$. We give a upper bound of instability of truncated symmetric powers $\mathrm{T}^l(\E)(0\leq l\leq\rk(\E)(p-1))$ in terms of $L_{\max}(\Omg^1_X)$, $\mathrm{I}(\Omg^1_X)$ and $\mathrm{I}(\E)$ (Theorem \ref{InstabTl}). As an application, We obtain a upper bound of Frobenius direct image ${F_X}_*(\E)$ and some sufficient conditions of slope semi-stability of ${F_X}_*(\E)$. In addition, we study the slope (semi)-stability of sheaves of locally exact (closed) forms $B^i_X$ ($Z^i_X$).
\end{abstract}

\section{Introduction}

Let $k$ be an algebraically closed field $k$ of characteristic $p>0$, $X$ a smooth projective variety of dimension $n$ over $k$ with a fixed ample divisor $H$. Let $\E$ be a torsion free sheaf on $X$, the slope of $\E$ is define as
$\mu(\E):=\frac{c_1(\E)\cdot H^{n-1}}{\rk(\E)}$.
Then $\E$ is called \emph{slop $($semi$)$-stable} if $\mu(\F)<(\leq)\mu(\E)$ for any nonzero subsheaf $\F\subsetneq\E$ with $\rk(\F)<\rk(\E)$.
For any torsion free sheaf $\E$, there exists a unique filtration
$$\mathrm{HN}_\bullet(\E):0=\mathrm{HN}_0(\E)\subset \mathrm{HN}_1(\E)\subset\cdots\subset \mathrm{HN}_{m-1}(\E)\subset \mathrm{HN}_m(\E)=\E$$
which is called \emph{Harder-Narasimhan filtration} of $\E$, such that
\begin{itemize}
    \item[$(a)$.] $gr^{\mathrm{HN}}_i(\E):=\mathrm{HN}_i(\E)/\mathrm{HN}_{i-1}(\E)(1\leq i\leq m)$ are slope semistable.
    \item[$(b)$.] $\mu_{\max}(\E):=\mu(gr^{\mathrm{HN}}_1(\E))>\cdots>\mu(gr^{\mathrm{HN}}_{m-1}(\E))>\mu(gr^{\mathrm{HN}}_m(\E))=:\mu_{\mathrm{min}}(\E)$.
\end{itemize}
The \emph{instability} of $\E$ was defined as rational number
$$\mathrm{I}(\E):=\mu_{\max}(\E)-\mu_{\mathrm{min}}(\E),$$
which is a numerical index measures how far a torsion free sheaf from being slope semi-stable, and $\E$ is slope semi-stable if and only if $\mathrm{I}(\E)=0$.

The absolute Frobenius morphism $F_X:X\rightarrow X$ is induced by homomorphism $\mathscr{O}_X\rightarrow \mathscr{O}_X$, $f\mapsto f^p$.
If for any integer $m\geq 0$, $m$-th Frobenius pullback $F^{m*}_X(\E)$ is slope (semi)-stable, then $\E$ is called a \emph{slope strongly $($semi$)$-stable sheaf}. In general, the slope (semi)-stability of torsion free sheaves does not preserved under Frobenius pull back $F^*_X$ (cf. \cite{Gieseker73}, \cite{Raynaud82}). However, V. Mehta, A. Ramanathan \cite[Theorem 2.1]{MehtaRamanathan83} showed that if $\mu_{\max}(\Omg^1_X)\leq 0$ then all slope semi-stable sheaves on $X$ are slope strongly semi-stable, and if $\mu_{\max}(\Omg^1_X)<0$ then all slope stable sheaves on $X$ are slope strongly stable. There are many classes of varieties satisfy $\mu_{\max}(\Omg^1_X)\leq 0$, such as homogeneous spaces, Abelian varieties, toric varieties and so on. The tensor product, exterior and symmetric products of slop strongly semi-stable sheaves are still slope strongly semi-stable in positive characteristic (cf. page 9 of \cite{Langer05}).

In general, A. Langer \cite[Theorem 2.7]{Langer04i} proved that for any torsion free sheaf $\E$ on $X$, there exists some integer $m_0\geq 0$ such that all quotients of the Harder-Narasimhan filtration of $F^{m*}_X(\E)$ are slope strongly semi-stable for any $m\geq m_0$. Moreover, A. Langer introduced the following definition $$L_{\max}(\E)=\lim\limits_{m\rightarrow\infty}\frac{\mu_{\max}({F_X^m}^*\E)}{p^m},~L_{\mathrm{min}}(\E)=\lim\limits_{m\rightarrow\infty}\frac{\mu_{\mathrm{min}}({F_X^m}^*\E)}{p^m},$$
and proved that \cite[Corollary 6.2]{Langer04i}
\begin{equation}\label{Lminmax}
L_{\max}(\E)-L_{\mathrm{min}}(\E)\leq\frac{\rk(\E)-1}{p}
\cdot\max\{0,~L_{\max}(\Omg^1_X)\}+\mathrm{I}(\E).
\end{equation}

On the other hand, Frobenius push-forwards also does not preserve the slope semi-stability of torsion free sheaves. However, for a smooth projective curve of genus $g\geq 2$, V. Mehta and C. Pauly \cite{MehtaPauly07} showed that Frobenius direct images of slope semi-stable sheaves are also slope semi-stable. At the same time, X. Sun \cite{Sun08} proved that on a smooth projective curve $X$ of genus $g\geq 1$, the Frobenius push-forwards preserves the slope semi-stability of locally free sheaves, and if genus $g\geq 2$ then Frobenius push-forwards preserves the slope stability of locally free sheaves via a different approach. In the case of higher dimension, X. Sun \cite[Corollary 4.9]{Sun08} showed that: For any torsion free sheaf $\E$ on $X$, if $\mu(\Omg^1_X)\geq 0$, then
\begin{equation}\label{InsDIFrob}
\mathrm{I}({F_X}_*(\E))\leq p^{n-1}\cdot\rk(\E)\cdot\max\{\mathrm{I}(\E\otimes_{\mathscr{O}_X}\mathrm{T}^l(\Omg^1_X))~|~0\leq l\leq n(p-1)\},
\end{equation}
where $\mathrm{T}^l(\Omg^1_X)(0\leq l\leq n(p-1))$ are truncated symmetric powers of $\Omg^1_X$, which are the associated bundles of $\Omg^1_X$ through some elementary representations of $\GL(n)$. Therefore, in order to bound the instability of Frobenius direct images, we should study the instabilities of $\mathrm{T}^l(\Omg^1_X)(0\leq l\leq n(p-1))$.

More generally, we study the instability of truncated symmetric powers of any torsion free sheaf $\E$. One of the main results of this paper is to give a upper bound of $\mathrm{I}(\mathrm{T}^l(\E))$ in terms of $L_{\max}(\Omg^1_X)$, $\mathrm{I}(\Omg^1_X)$ and $\mathrm{I}(\E)$ (Theorem \ref{InstabTl}).

As an application, we obtain a upper bound of Frobenius direct image ${F_X}_*(\E)$ in terms of $L_{\max}(\Omg^1_X)$, $\mathrm{I}(\Omg^1_X)$ and $\mathrm{I}(\E)$ (Theorem \ref{InstabDirIm}) and some sufficient conditions of slope semi-stability of ${F_X}_*(\E)$ (when $\mu(\Omg^1_X)\geq 0$) (Proposition \ref{FroDirIm}). In particular, when the cotangent sheaf $\Omg^1_X$ is slope strongly semi-stable with $\mu(\Omg^1_X)\geq 0$, then the slope strongly semi-stability of $\E$ implies slope semi-stability of ${F_X}_*(\E)$.

In addition, we study the slope (semi)-stability of the sheaves of locally exact (closed) differential $i$-forms $B^i_X$ ($Z^i_X$) on the higher dimensional smooth projective variety in positive characteristic. X. Sun \cite{Sun10ii} showed that when $X$ is a smooth projective surface such that $\Omg^1_X$ is slope semi-stable with $\mu(\Omg^1_X)>0$, then $Z^1_X$ is slope semi-stable if $\mathrm{char}(k)=3$ and $Z^1_X$ is slope stable if $\mathrm{char}(k)>3$. However, we show that $Z^i_X(1\leq i<\frac{n}{2})$ are never slope semi-stable if $n\geq 3$ and $\mu(\Omg^1_X)>0$ (Proposition \ref{Prop:InstZiX}). Moreover, we show that if $\Omg^1_X$ is slope semi-stable with $\mu(\Omg^1_X)=0$, then $B^i_X$ and $Z^i_X$ are slope strongly semi-stable for any $1\leq i\leq n$. At last, we prove that if $\mathrm{T}^l(\Omg^1_X)(0\leq l\leq n(p-1))$ are slope semi-stable, then $B^n_X$ is slope strongly semi-stable if $\mu(\Omg^1_X)=0$ and $B^n_X$ is slope stable if $\mu(\Omg^1_X)>0$. This generalize the result of X. Sun \cite[Lemma 3.2]{Sun10ii}.

In this paper, unless otherwise explicitly declared, $k$ is an algebraically closed field of characteristic $p>0$, and $X$ is a smooth projective variety of dimension $n$ over $k$ with a fixed ample divisor $H$.

\section{Preliminaries}

The canonical filtration was first introduced by Joshi-Ramanan-Xia-Yu in \cite{JRXY06} in the curve case. X. Sun \cite{Sun08} generalized the canonical filtration to higher dimensional case and used this to study slope (semi)-stability of Frobenius direct images.

\begin{Definition}\label{CanFil}
Let $\E$ be a coherent sheaf on $X$,
$$\nabla_{\mathrm{can}}:F^*_X{F_X}_*(\E)\rightarrow F^*_X{F_X}_*(\E)\otimes_{\mathscr{O}_X}\Omg^1_X$$ the canonical connection of $F^*_X{F_X}_*(\E)$ (cf. \cite{Katz70} Section 5). Set
\begin{eqnarray*}
V_1&:=&\Ker(F^*_X{F_X}_*(\E)\twoheadrightarrow\E)),\\
V_{l+1}&:=&\Ker\{V_l\stackrel{\nabla}{\rightarrow}F^*_X{F_X}_*(\E)\otimes_{\mathscr{O}_X}\Omg^1_X\rightarrow (F^*_X{F_X}_*(\E)/V_l)\otimes_{\mathscr{O}_X}\Omg^1_X\}.
\end{eqnarray*}
The filtration ${\mathbb{F}^{\mathrm{can}}_\E}_\bullet:F^*_X{F_X}_*(\E)=V_0\supset V_1\supset V_2\supset\cdots$,
is called the \emph{canonical filtration} of $F^*_X{F_X}_*(\E)$.
\end{Definition}

\begin{Theorem}\cite[Theorem 3.7]{Sun08}\label{CanonicalFiltration}
Let $\E$ be a locally free sheaf on $X$. Then the canonical filtration of
$F^*_X{F_X}_*(\E)$ is\\
\centerline{$0=V_{n(p-1)+1}\subset V_{n(p-1)}\subset\cdots\subset V_1\subset V_0=F^*_X{F_X}_*(\E)$} such that
$\nabla^l:V_l/V_{l+1}\cong\E\otimes_{\mathscr{O}_X}\mathrm{T}^l(\Omg^1_X),~0\leq
l\leq n(p-1)$, where $\mathrm{T}^l(\Omg^1_X)$ are the truncated symmetric powers of $\Omg^1_X$.
\end{Theorem}

We recall the construction and properties of truncated symmetric powers of vector spaces. Let $V$ is a $n$-dimensional vector space over $k$ with standard representation of $\GL_n(k)$, $S_l$ the symmetric group of $l$ elements with a natural action on $V^{\otimes l}$ by $(v_1\otimes\cdots\otimes v_l)\cdot\sigma=v_{\sigma_{(1)}}\otimes\cdots\otimes v_{\sigma_{(l)}}$ for $v_i\in V$ and $\sigma\in S_l$.
Let $e_1,\cdots,e_n$ be a basis of $V$. For any non-negative partition $(k_1,\cdots,k_n)$ of $l$ (i.e. $l=\sum\limits^n_{i=1}k_i$, $k_i\geq 0,~1\leq i\leq n$), define $v(k_1,\cdots,k_n):=\sum\limits_{\sigma\in S_l}(e^{\otimes k_1}_1\otimes\cdots\otimes e^{\otimes k_n}_n)\cdot\sigma$.
Let $\mathrm{T}^l(V)\subset V^{\otimes l}$ be the linear subspace generated by all vectors\\ \centerline{$\{v(k_1,\cdots,k_n)~|~l=\sum\limits^n_{i=1}k_i,~k_i\geq 0,~1\leq i\leq n\}$.}
Then $\mathrm{T}^l(V)$ is a $\GL_n(k)$ sub-representation of $V^{\otimes l}$. Let $F^*_kV$ be the Frobenius twist of the standard representation $V$ of $\GL_n(k)$ through the homomorphism $\GL_n(k)\rightarrow \GL_n(k):((a_{ij})_{n\times n}\mapsto(a^p_{ij})_{n\times n})$. Define $k$-linear maps
\begin{equation}\label{Phiq}
\phi_q:\mathrm{Sym}^{l-q\cdot p}(V)\otimes_k\bigwedge^q(F^*_kV)\rightarrow \mathrm{Sym}^{l-(q-1)\cdot p}(V)\otimes_k\bigwedge^{q-1}(F^*_kV)
\end{equation}
\centerline{$f\otimes e_{k_1}\wedge\cdots\wedge e_{k_q}\mapsto\sum\limits^q_{i=1}(-1)^{i-1}e^p_{k_i}f\otimes e_{k_1}\wedge\cdots\wedge\widehat{e}_{k_i}\wedge\cdots\wedge e_{k_q}$}

\begin{Lemma}\cite[Proposition 3.5]{Sun08}\label{RepT}
There exists an exact sequence of $\GL_n(k)$-representations\\
\centerline{$0\rightarrow \mathrm{Sym}^{l-l(p)\cdot p}(V)\otimes_k\bigwedge\limits^{l(p)}F^*_k(V)\stackrel{\phi_{l(p)}}{\rightarrow}\mathrm{Sym}^{l-(l(p)-1)\cdot p}(V)\otimes_k\bigwedge\limits^{l(p)-1}F^*_k(V)\rightarrow$}
\centerline{$\cdots\rightarrow \mathrm{Sym}^{l-q\cdot p}(V)\otimes_k\bigwedge\limits^{q}F^*_k(V)\stackrel{\phi_{q}}{\rightarrow}\mathrm{Sym}^{l-(q-1)\cdot p}(V)\otimes_k\bigwedge\limits^{q-1}F^*_k(V)\rightarrow\cdots$}
\centerline{$\rightarrow \mathrm{Sym}^{l-p}(V)\otimes_kF^*_k({V})\stackrel{\phi_{1}}{\rightarrow} \mathrm{Sym}^l(V)\stackrel{\phi_0}{\rightarrow}\mathrm{T}^l(V)\rightarrow 0.$}
\end{Lemma}

Let $\E$ be a locally free sheaf of rank $n$ on $X$, $\mathrm{T}^l(\E)\subset\E^{\otimes l}$ is defined to be the sheaf of sections of the associated vector bundle of the frame bundle of $\E$ (principal $\GL_n(k)$-bundle) through the representation $\mathrm{T}^l(V)$.

\section{Instability of Truncated Symmetric Powers}

Let we recall the definition and some properties of Harder-Narasimhan polygons of sheaves, which was first introduced by  S. S. Shatz in \cite{Shatz77}.

For any coherent sheaf $\F$ on $X$, we may associate to it the point $\mathrm{p}(\rk(\F),\dg(\F))$ in the plane with coordinates rank and degree.
Let $\E$ be a torsion free sheaf,
$$\mathrm{HN}_\bullet(\E):0=\E_0\subset \E_1\subset\cdots\subset \E_{m-1}\subset \E_m=\E$$
the Harder-Narasimhan filtration of $\E$. Consider the points
$$\mathrm{p}(\rk(\E_i),\dg(\E_i))(1\leq i\leq m)$$ in the coordinate plane $\mathrm{rank}$-$\dg$, and connect point $\mathrm{p}(\rk(\E_i),\dg(\E_i))$ to
point $\mathrm{p}(\rk(\E_{i+1}),\dg(\E_{i+1}))$ successively by line segments. Then we get a polygon in the plane which we call the \emph{Harder-Narasimhan Polygon} of $\E$, denote by $\mathrm{HNP}(\E)$.

Let $\mathrm{p}(r,d)$ be a point in the coordinate plane of $\mathrm{rank}$-$\dg$. If $r\leq\rk(\E)$, and $d\geq(\leq)d^\prime$ for some point $\mathrm{p}(r,d^\prime)\in\mathrm{HNP}(\E)$, then we say $\mathrm{p}(r,d)$ \emph{lies on $($below$)$} the $\mathrm{HNP}(\E)$. In this sense, there is a natural partial order structure, denote by "$\succcurlyeq$", on the set $\{\mathrm{HNP}(\E)~|~\E\in\mathfrak{Coh}(X)~\}$.

\begin{Lemma}[S. S. Shatz \cite{Shatz77} Theorem 3]\label{Shatz77}
Let $k$ be an algebraically closed field, $S$ a scheme of finite type over $k$, $\mathcal{E}$ a flat family of torsion free sheaves on $S\times_kX$ parameterized by $S$.
Then for any convex polygon $\mathscr{P}$, subset $S_{\mathscr{P}}=\{s\in S|\mathrm{HNP}(\mathcal{E}|_s)\succcurlyeq\mathscr{P}\}$ is a closed subset of $S$.
\end{Lemma}

Let $\E$ a torsion free sheaf on $X$. Then there exists an open subset $U\subseteq X$, such that $\mathrm{codim}_X(X-U)\geq 2$ and $\E|_U$ is locally free.
Let
$$\widehat{\mathrm{T}^l(\E)}_U:={i_U}_*(\mathrm{T}^l(\E|_U))$$
where $i_U:U\rightarrow X$ is the natural open immersion. Then $\widehat{\mathrm{T}^l(\E)}_U$ is a reflective torsion free sheaf such that $\mu(\widehat{\mathrm{T}^l(\E)}_U)$ and $\mathrm{I}(\widehat{\mathrm{T}^l(\E)}_U)$ are independent of the choice of $U$. Without loss of generality, for any torsion free sheaf $\E$, we denote
$$\mu(\mathrm{T}^l(\E)):=\mu(\widehat{\mathrm{T}^l(\E)}_U),~ \mathrm{I}(\mathrm{T}^l(\E)):=\mathrm{I}(\widehat{\mathrm{T}^l(\E)}_U),$$
for any open subset $U\subseteq X$ such that $\mathrm{codim}_X(X-U)\geq 2$ and $\E|_U$ is locally free.

\begin{Lemma}\label{Lem:DIFlat}
Let $R$ be a principal ideal domain, $S=\mathrm{Spec}(R)$, and $X$ an integral scheme over $S$, $U$ an open subscheme of $X$. Let $\E$ be a torsion free sheaf on $U$ which is flat over $S$ such that ${i_U}_*(\E)$ is a coherent sheaf on $X$, where $i_U:U\rightarrow X$ is the natural open immersion. Then ${i_U}_*(\E)$ is flat over $S$.
\end{Lemma}
\begin{proof} Since $R$ is a principal ideal domain, to prove ${i_U}_*(\E)$ is flat over $S$, it is enough to show that $M:=({i_U}_*(\E))(V)$ is a torsion free $R$-module for any affine open subset $V\subseteq X$. Let $0\neq r\in R$, $0\neq m\in M$. As $X$ is an integral scheme and $\E$ is torsion free over $U$, then $m|_W\neq 0$ for any affine open sub-set $W\subseteq U\cap V$. On the other hand, $r\cdot(m|_W)\neq 0$ for $\E$ is flat over $S$, therefore $rm\neq 0$. Thus $M$ is a torsion free $R$-module.
\end{proof}

\begin{Proposition}\label{Tl1}
Let $\E$ be a torsion free sheaf on $X$, $0=\E_0\subset\E_1\subset\cdots\subset\E_{m-1}\subset\E_m=\E$ a filtration of $\E$ such that $\E_i/\E_{i-1}(1\leq i\leq m)$ are torsion free sheaves.
Then for any integer $0\leq l\leq\rk(\E)(p-1)$, we have
$$\mathrm{I}(\mathrm{T}^l(\E))\leq\mathrm{I}(\mathrm{T}^l(\bigoplus_{i=1}^{m}\E_i/\E_{i-1})).$$
\end{Proposition}

\begin{proof} We use induction on $m$, so we only have to prove the case $m=2$. In fact, for any $1\leq i\leq m$, consider the exact sequence\\
\centerline{$0\longrightarrow \E_{i-1}\bigoplus(\bigoplus\limits_{j=i+1}^m\E_j/\E_{j-1})\longrightarrow \E_i\bigoplus(\bigoplus\limits_{j=i+1}^m\E_j/\E_{j-1}) \longrightarrow \E_i/\E_{i-1} \longrightarrow 0,$}
we have $\mathrm{I}(\mathrm{T}^l(\E))\leq\mathrm{I}(\mathrm{T}^l(\E_i\bigoplus(\bigoplus\limits_{j=i+1}^m\E_j/\E_{j-1})))
\leq\mathrm{I}(\mathrm{T}^l(\bigoplus\limits_{i=1}^{m}\E_i/\E_{i-1}))$.

In general, for the exact sequence of torsion free sheaves\\
\centerline{$0\longrightarrow\E^\prime\longrightarrow\E\longrightarrow \E^{\prime\prime}\longrightarrow0,$}
we can construct a coherent sheaf $\mathcal{E}$ on $\mathbb{A}^1_k\times X$ which is flat over $\mathbb{A}^1_k$, such that $\mathcal{E}|_{t\times X}\cong\E$ for any $t\in\mathbb{A}^1_k-\{0\}$, and $\mathcal{E}|_{\{0\}\times X}\cong\E^\prime\bigoplus\E^{\prime\prime}$.
Choose an open subset $U\subseteq X$ with $\mathrm{codim}_X(X-U)\geq 2$ such that $\E|_U$, $\E^\prime|_U$ and $\E^{\prime\prime}|_U$ are locally free. Then for any $t\in\mathbb{A}^1_k$, $\mathcal{E}|_{\{t\}\times U}$ is locally free on $U$. Hence, by \cite[Theorem 1.27]{Simpson94} we have $\mathcal{E}|_{\mathbb{A}^1_k\times U}$ is locally free on $\mathbb{A}^1_k\times U$.

By Lemma \ref{RepT} we can construct an exact sequence of sheaves on $Y:=\mathbb{A}^1_k\times U$\\
\centerline{$0\rightarrow \mathrm{Sym}^{l-l(p)\cdot p}(\mathcal{E}|_Y)\otimes_{\mathscr{O}_Y}\bigwedge\limits^{l(p)}F^*_Y(\mathcal{E}|_Y)\stackrel{\phi_{l(p)}}{\rightarrow}
    \mathrm{Sym}^{l-(l(p)-1)\cdot p}(\mathcal{E}|_Y)\otimes_{\mathscr{O}_Y}\bigwedge\limits^{l(p)-1}F^*_Y(\mathcal{E}|_Y)\rightarrow$}
\centerline{$\cdots\rightarrow \mathrm{Sym}^{l-q\cdot p}(\mathcal{E}|_Y)\otimes_{\mathscr{O}_Y}\bigwedge\limits^{q}F^*_Y(\mathcal{E}|_Y)\stackrel{\phi_{q}}{\rightarrow}\mathrm{Sym}^{l-(q-1)\cdot p}(\mathcal{E}|_Y)\otimes_{\mathscr{O}_Y}\bigwedge\limits^{q-1}F^*_Y(\mathcal{E}|_Y)\rightarrow\cdots$}
\centerline{$\rightarrow \mathrm{Sym}^{l-p}(\mathcal{E}|_Y)\otimes_{\mathscr{O}_Y}F^*_Y({\mathcal{E}|_Y})\stackrel{\phi_{1}}{\rightarrow} \mathrm{Sym}^l(\mathcal{E}|_Y)\stackrel{\phi_0}{\rightarrow}\mathrm{T}^l(\mathcal{E}|_Y)\rightarrow 0.$}

Let $$\widehat{\mathrm{T}^l(\mathcal{E})}:={i_{\mathbb{A}^1_k\times U}}_*(\mathrm{T}^l(\mathcal{E}|_{\mathbb{A}^1_k\times U}))$$
where $i_{\mathbb{A}^1_k\times U}:\mathbb{A}^1_k\times U\rightarrow\mathbb{A}^1_k\times X$ is the natural open immersion.
Since $\mathbb{A}^1_k\times X$ is a smooth variety and $\mathrm{codim}_{\mathbb{A}^1_k\times X}(\mathbb{A}^1_k\times X-\mathbb{A}^1_k\times U)\geq 2$, $\mathrm{T}^l(\mathcal{E}|_{\mathbb{A}^1_k\times U})~\text{is locally free}$, by \cite[Proposition 5.10]{Grothendieck67} and Lemma \ref{Lem:DIFlat}, we have $\widehat{\mathrm{T}^l(\mathcal{E})}$ is a coherent sheaf on $\mathbb{A}^1_k\times X$ which is flat over $\mathbb{A}^1_k$ such that $\widehat{\mathrm{T}^l(\mathcal{E})}_t|_U\cong\mathrm{T}^l(\E|_U)$ for any $t\in\mathbb{A}^1_k-\{0\}$ and $\widehat{\mathrm{T}^l(\mathcal{E})}_0|_U\cong\mathrm{T}^l((\E^\prime\bigoplus\E^{\prime\prime})|_U)$.
Then by Lemma \ref{Shatz77}, we have
$$\mathrm{I}(\mathrm{T}^l(\E))=\mathrm{I}(\widehat{\mathrm{T}^l(\mathcal{E})}_t)\leq\mathrm{I}(\widehat{\mathrm{T}^l(\mathcal{E})}_0)=\mathrm{I}(\mathrm{T}^l(\E^\prime\oplus\E^{\prime\prime})).$$
for any $t\in\mathbb{A}^1_k-\{0\}$. This completes the proof.
\end{proof}

The following theorem shows that when $\E$ is a direct sum of slope strongly semi-stable sheaves, we can bound $\mathrm{I}(\mathrm{T}^l(\E))$ in terms of $\rk(\E)$ and $\mathrm{I}(\E)$.

\begin{Theorem}\label{Tl2}
Let $\E=\E_1\oplus\cdots\oplus\E_m$ be a torsion free sheaf such that $\E_i~(1\leq i\leq m)$ are slope strongly semi-stable. Then for any integer $0\leq l\leq\rk(\E)(p-1)$, we have
$$\mathrm{I}(\mathrm{T}^l(\E))\leq\mathrm{min}\{~l,~[\frac{\rk(\E)}{2}](p-1)~\}\cdot\mathrm{I}(\E).$$
\end{Theorem}

\begin{proof} Without loss of generality, we can assume $\E_i(1\leq i\leq m)$ are locally free such that $\mu(\E_1)>\mu(\E_2)>\cdots>\mu(\E_m)$. Let $r=\rk(\E)$, $r_i=\rk(\E_i)(1\leq i\leq m)$, $l(p)$ be the unique integer such that $0\leq l-l(p)\cdot p<p$.

Since tensor products, exterior and symmetric products of slope strongly semi-stable sheaves are still slope strongly semi-stable, the direct summand\\
\centerline{$\E^{a_1,\cdots, a_m}_{b_1,\cdots, b_m}:=\mathrm{Sym}^{a_1}\E_1\otimes\cdots\otimes\mathrm{Sym}^{a_m}\E_m\otimes\bigwedge\limits^{b_1}F_X^*\E_1\otimes\cdots\otimes\bigwedge\limits^{b_m}F_X^*\E_m $}
of $\mathrm{Sym}^{a}\E\otimes_{\mathscr{O}_X}\bigwedge\limits^{b}F_X^*\E$ is slope strongly semi-stable for any non-negative partition $(a_1,\cdots, a_m)$, $(b_1,\cdots, b_m)$ of non-negative integers $a$ and $b$.
By direct computation, if $\E^{a_1,\cdots, a_m}_{b_1,\cdots, b_m}\neq\emptyset$, then
$\mu(\E^{a_1,\cdots, a_m}_{b_1,\cdots, b_m})=(a_1+pb_1)\mu(\E_1)+\cdots+(a_m+pb_m)\mu(\E_m)$.

By Lemma \ref{RepT}, we obtain the exact sequence of locally free sheaves\\
\centerline{$0\rightarrow \mathrm{Sym}^{l-l(p)\cdot p}(\E)\otimes_{\mathscr{O}_X}\bigwedge\limits^{l(p)}F^*_X(\E)\stackrel{\phi_{l(p)}}{\rightarrow}
    \mathrm{Sym}^{l-(l(p)-1)\cdot p}(\E)\otimes_{\mathscr{O}_X}\bigwedge\limits^{l(p)-1}F^*_X(\E)\rightarrow$}
\centerline{$\cdots\rightarrow \mathrm{Sym}^{l-q\cdot p}(\E)\otimes_{\mathscr{O}_X}\bigwedge\limits^{q}F^*_X(\E)\stackrel{\phi_{q}}{\rightarrow}\mathrm{Sym}^{l-(q-1)\cdot p}(\E)\otimes_{\mathscr{O}_X}\bigwedge\limits^{q-1}F^*_X(\E)\rightarrow\cdots$}
\begin{equation}\label{ExactSeq}
\rightarrow \mathrm{Sym}^{l-p}(\E)\otimes_{\mathscr{O}_X}F^*_X({\E})\stackrel{\phi_{1}}{\rightarrow} \mathrm{Sym}^l(\E)\stackrel{\phi_0}{\rightarrow}\mathrm{T}^l(\E)\rightarrow 0.
\end{equation}

By the definition of $\phi_q$ in section 2 (See (\ref{Phiq})), we have
$$\phi_q(\E^{a_1,\cdots, a_m}_{b_1,\cdots, b_m})\subseteq\bigoplus\limits_{i=1}^m\E^{a_1,\cdots,a_i+p, \cdots a_m}_{b_1,\cdots, b_i-1,\cdots, b_m}.$$
where $l=\sum\limits_{i=1}^m(a_i+pb_i)$. As $\mu(\E^{a_1,\cdots,a_i+p,\cdots, a_m}_{b_1,\cdots, b_i-1,\cdots, b_m})=\bigoplus\limits_{i=1}^m(a_i+pb_i)\mu(\E_i)$ for any $1\leq i\leq m$.
Thus $\mu(\E^{a_1,\cdots, a_m}_{b_1,\cdots, b_m})=\mu(\bigoplus\limits_{i=1}^m\E^{a_1,\cdots,a_i+p,\cdots a_m}_{b_1,\cdots,b_i-1,\cdots, b_m})$. Let $(\mathrm{Sym}^{l-q\cdot p}(\E)\otimes_{\mathscr{O}_X}\bigwedge\limits^qF^*_X(\E))_{\mu}$ be the direct summand of $\mathrm{Sym}^{l-q\cdot p}(\E)\otimes_{\mathscr{O}_X}\bigwedge^{q}F^*_X(\E)$ with slope $\mu$, where $\mu=\sum\limits_{i=1}^mc_i\mu(\E_i),~l=\sum\limits_{i=1}^mc_i,~c_i\geq 0(0\leq i\leq m)$ and $0\leq q\leq l(p)$.
Then the exact sequence (\ref{ExactSeq}) can be decomposed into direct summands\\
\centerline{$0\rightarrow (\mathrm{Sym}^{l-l(p)\cdot p}(\E)\otimes_{\mathscr{O}_X}\bigwedge\limits^{l(p)}F^*_X(\E))_{\mu}\stackrel{\phi_{l(p)}}{\rightarrow}
(\mathrm{Sym}^{l-(l(p)-1)\cdot p}(\E)\otimes_{\mathscr{O}_X}\bigwedge\limits^{l(p)-1}F^*_X(\E))_{\mu}\rightarrow$}
\centerline{$\cdots\rightarrow (\mathrm{Sym}^{l-q\cdot p}(\E)\otimes_{\mathscr{O}_X}\bigwedge\limits^{q}F^*_X(\E))_{\mu}\stackrel{\phi_{q}}{\rightarrow}(\mathrm{Sym}^{l-(q-1)\cdot p}(\E)\otimes_{\mathscr{O}_X}\bigwedge\limits^{q-1}F^*_X(\E))_{\mu}\rightarrow\cdots$}
\centerline{$\rightarrow (\mathrm{Sym}^{l-p}(\E)\otimes_{\mathscr{O}_X}F^*_X({\E}))_{\mu}\stackrel{\phi_{1}}{\rightarrow} (\mathrm{Sym}^l(\E))_{\mu}\stackrel{\phi_{0}}{\rightarrow}(\mathrm{T}^l(\E))_{\mu}\rightarrow 0,$}
and $$\mathrm{T}^l(\E)=\bigoplus\limits_{\mu=\sum\limits_{i=1}^mc_i\mu(\E_i), \atop \sum\limits_{i=1}^mc_i=l,c_i\geq0(1\leq i\leq m)}(\mathrm{T}^l(\E))_{\mu}.$$

Consider the direct sum $\mathrm{Sym}^l(\E)=\bigoplus\limits_{\sum\limits_{i=1}^mk_i=l,\atop k_i\geq 0,1\leq i\leq m}\bigotimes\limits_{i=1}^m \mathrm{Sym}^{k_i}(\E_i)$, we have $$\phi_0(\bigotimes\limits_{i=1}^m \mathrm{Sym}^{k_i}(\E_i))\neq 0~\text{if and only if}~0\leq k_i<r_i(p-1)(1\leq i\leq m).$$
Let $(d_1,\cdots,d_{l(p)},\cdots,d_r):=(p-1,\cdots,p-1,l-l(p)\cdot p,\cdots,0)$. Then
$$\mu_{\max}(\mathrm{T}^l(\E))=\sum^{r}_{i=1} d_i\mu(\E_i),~\mu_{\mathrm{min}}(\mathrm{T}^l(\E))=\sum^{r}_{i=1} d_{r-i+1}\mu(\E_i).$$
\begin{eqnarray*}\mathrm{I}(\mathrm{T}^l(\E))&=&\sum^{r}_{i=1}(d_i-d_{r-i+1})\mu(\E_i)=\sum^{[\frac{r}{2}]}_{i=1}(d_i-d_{r-i+1})(\mu(\E_i)-\mu(\E_{r-i+1}))\\
&\leq&\sum^{[\frac{r}{2}]}_{i=1}(d_i-d_{r-i+1})(\mu_{\max}(\E)-\mu_{\mathrm{min}}(\E)).
\end{eqnarray*}
\[=\left\{
\begin{array}{lll}
l\cdot\mathrm{I}(\E) & \mathrm{if}~0\leq l\leq [\frac{r}{2}](p-1)\\
(r(p-1)-l)\cdot\mathrm{I}(\E) & \mathrm{if}~r|2, \frac{r}{2}(p-1)<l<r(p-1)\\
(p-1)\cdot[\frac{r}{2}]\cdot\mathrm{I}(\E) & \mathrm{if}~r\nmid 2, [\frac{r}{2}](p-1)<l\leq([\frac{r}{2}]+1)(p-1)\\
(r(p-1)-l)\cdot\mathrm{I}(\E) & \mathrm{if}~r\nmid 2, ([\frac{r}{2}]+1)(p-1)<l<r(p-1)\\
\end{array}
\right.
\]

Hence, $\mathrm{I}(\mathrm{T}^l(\E))\leq\mathrm{min}\{~l,~[\frac{\rk(\E)}{2}](p-1)~\}\cdot\mathrm{I}(\E)$, for any $0\leq l\leq\rk(\E)(p-1)$.
\end{proof}

\begin{Theorem}\label{InstabTl}
Let $\E$ be a torsion free sheaf on $X$. Then for any integer $0\leq l\leq\rk(\E)(p-1)$, we have
$$\mathrm{I}(\mathrm{T}^l(\E))\leq\mathrm{min}\{~l,~[\frac{\rk(\E)}{2}](p-1)~\}\cdot(\frac{\rk(\E)-1}{p}
\cdot\max\{0,~L_{\max}(\Omg^1_X)\}+\mathrm{I}(\E)).$$
\end{Theorem}

\begin{proof} By \cite[Theorem 2.7]{Langer04i}, there exists an integer $m_0\geq 0$ such that the Harder-Narasimhan filtration $0\subset\E_1\subset\cdots\subset\E_{m-1}\subset\E_m={F_X^{m_0}}^*(\E)$ of ${F_X^{m_0}}^*(\E)$ satisfy $gr^{\mathrm{HN}}_i({F_X^{m_0}}^*(\E))=\E_i/\E_{i-1}(1\leq i\leq m)$ are slope strongly semi-stable. Thus,
\begin{eqnarray*}\mathrm{I}(\mathrm{T}^l(\E))&\leq&\frac{\mathrm{I}(F_X^{m_0*}(\mathrm{T}^l(\E))}{p^{m_0}}\\
&=&\frac{\mathrm{I}(\mathrm{T}^l(F_X^{m_0*}(\E)))}{p^{m_0}}\\
&\leq&\frac{\mathrm{I}(\mathrm{T}^l(\bigoplus\limits_{i=1}^m\E_i/\E_{i-1}))}{p^{m_0}}~~(\mathrm{cf.~Theorem}~\ref{Tl1})\\
&\leq&\mathrm{min}\{~l,~[\frac{\rk(\E)}{2}](p-1)~\}\cdot\frac{\mathrm{I}(\bigoplus\limits_{i=1}^m\E_i/\E_{i-1})}{p^{m_0}} ~~(\mathrm{cf.~Theorem}~\ref{Tl2})\\
&=&\mathrm{min}\{~l,~[\frac{\rk(\E)}{2}](p-1)~\}\cdot(L_{\max}(\E)-L_{\mathrm{min}}(\E))\\
&\leq&\mathrm{min}\{~l,~[\frac{\rk(\E)}{2}](p-1)~\}\cdot(\frac{\rk(\E)-1}{p}
\cdot\max\{0,~L_{\max}(\Omg^1_X)\}+\mathrm{I}(\E)).
\end{eqnarray*}
where the first inequality is followed by the property of Frobenius morphism, and the last inequality is just the inequality (\ref{Lminmax}).
\end{proof}

We give some sufficient conditions for slope semi-stability of truncated symmetric powers of torsion free sheaves as following.
\begin{Proposition}\label{SemiStabT^l}
Let $\E$ be a slope strongly semi-stable torsion free sheaf on $X$. Then $\mathrm{T}^l(\E)$ is slope strongly semi-stable for any integer $0\leq l\leq\rk(\E)(p-1)$.
\end{Proposition}

\begin{proof} Without loss of generality, we can assume $\E$ is locally free. For any integer $0\leq l\leq\rk(\E)(p-1)$, consider the exact sequence\\
\centerline{$0\longrightarrow \mathrm{Sym}^{l-l(p)\cdot p}(\E)\otimes_{\mathscr{O}_X}
\bigwedge\limits^{l(p)}F_X^*(\E)\longrightarrow
\mathrm{Sym}^{l-(l(p)-1)\cdot p}(\E)\otimes_{\mathscr{O}_X}\bigwedge\limits^{l(p)-1}
F_X^*(\E)\longrightarrow\cdots$}
\centerline{$\longrightarrow\mathrm{Sym}^{l-p}(\E)\otimes_{\mathscr{O}_X} F_X^*(\E)\longrightarrow
\mathrm{Sym}^l(\E)\longrightarrow \mathrm{T}^l(\E)\longrightarrow 0,$}
we have $\mu(\mathrm{Sym}^{l-i\cdot p}(\E)\otimes_{\mathscr{O}_X}\bigwedge\limits^iF_X^*(\E))=\mu(\mathrm{T}^l(\E))$ for any integer $0\leq i\leq l(p)$ by direct computation. Since $\mathrm{Sym}^{l-i\cdot p}(\E)\otimes_{\mathscr{O}_X}\bigwedge\limits^iF_X^*(\E)(0\leq i\leq l(p))$ are slope strongly semi-stable. Thus $\mathrm{T}^l(\E)$ is also slope strongly semi-stable by the trivial remark: For any short exact sequence of torsion free sheaves with same slope, the middle term is slope (strongly) semi-stable if and only if the other two terms are slope (strongly) semi-stable.
\end{proof}

\section{Instability of Frobenius Direct Images}

In this section, we study the instability of Frobenius direct images of torsion free sheaves in two cases. $(\mathrm{I})$. $\Omg^1_X$ is slope semi-stable with $\mu(\Omg^1_X)\leq0$; $(\mathrm{II})$. $\mu(\Omg^1_X)\geq 0$.

\subsection{Case I: $\Omg^1_X$ is slope semi-stable with $\mu(\Omg^1_X)\leq0$}

\begin{Theorem}\label{Thm:DiIm-}
If $\Omg^1_X$ is slope semi-stable with $\mu(\Omg^1_X)\leq0$. Then for any slope semi-stable sheaf $\E$ on $X$, we have
$$\mathrm{I}({F_X}_*(\E))\leq-\frac{n(p-1)p^{n-1}\cdot\rk(\E)\cdot\mu(\Omg^1_X)}{2}.$$
\end{Theorem}
\begin{proof} Since $\Omg^1_X$ is slope semi-stable with $\mu(\Omg^1_X)\leq0$, hence $\E\otimes_{\mathscr{O}_X}\mathrm{T}^l(\Omg^1_X)$ is slope semi-stable for any $0\leq l\leq n(p-1)$. Let $0\subset V_{n(p-1)}\subset\cdots\subset V_1\subset V_0=F^*_X{F_X}_*(\E)$ be the canonical filtration of $F^*_X{F_X}_*(\E)$. For any sub-sheaf $\F$ of ${F_X}_*(\E)$, let $m=\mathrm{max}\{~l~|~V_m\cap F^*_{X/k}(\E)\neq 0~\}$ and
$$\E_l:=\frac{V_l\cap F^*_{X/k}(\E)}{V_{l+1}\cap F^*_{X/k}(\E)}\subset\frac{V_l}{V_{l+1}},~r_l:=\mathrm{rk}(\E_l),~0\leq l\leq m.$$
Then by \cite[Lemma 4.4]{Sun08}, we have
\begin{eqnarray*}
   \mu(\F)-\mu({F_X}_*(\E))&\leq&-\frac{\mu(\Omg^1_X)}{p\cdot\rk(\F)}\sum\limits^m_{l=0}(\frac{n(p-1)}{2}-l)\cdot r_l\\
   &\leq&-\frac{n(p-1)\cdot\mu(\Omg^1_X)}{2\cdot p}.
\end{eqnarray*}
Thus $$\mathrm{I}({F_X}_*(\E))\leq-\frac{n(p-1)\cdot\mu(\Omg^1_X)}{2\cdot p}\cdot\rk({F_X}_*(\E))=-\frac{n(p-1)p^{n-1}\cdot\rk(\E)\cdot\mu(\Omg^1_X)}{2}.$$
This is follows from the fact that for any torsion free sheaf $\F$, if there is a constant $\lambda$ such that
$\mu(\mathscr{G})-\mu(\F)\leq\lambda$ for any subsheaf $\mathscr{G}\subset\F$. Then $\mathrm{I}(\F)\leq\rk(\F)\cdot\lambda$.
\end{proof}

\subsection{Case II: $\mu(\Omg^1_X)\geq 0$}

\begin{Proposition}\label{Tensor}
Let $\E_i(1\leq i\leq m)$ be torsion free sheaves. Then
$$\mathrm{I}(\bigotimes_{i=1}^m\E_i)\leq\frac{\sum\limits_{i=1}^m\rk(\E_i)-m}{p}
\cdot\max\{0,~L_{\max}(\Omg^1_X)\}+\sum_{i=1}^m\mathrm{I}(\E_i).$$
\end{Proposition}

\begin{proof} Since $L_{\max}(\bigotimes\limits_{i=1}^m\E_i)=\sum\limits_{i=1}^mL_{\max}(\E_i)$, we have
$$L_{\mathrm{min}}(\bigotimes\limits_{i=1}^m\E_i)=-L_{\max}((\bigotimes\limits_{i=1}^m\E_i)^{\vee})
=-L_{\max}(\bigotimes\limits_{i=1}^m\E_i^{\vee})=\sum\limits_{i=1}^mL_{\mathrm{min}}(\E_i).$$
Hence,
\begin{eqnarray*}
\mathrm{I}(\bigotimes\limits_{i=1}^m\E_i)&\leq&L_{\max}(\bigotimes\limits_{i=1}^m\E_i)-L_{\mathrm{min}}(\bigotimes\limits_{i=1}^m\E_i)=\sum\limits_{i=1}^m(L_{\max}(\E_i)-L_{\mathrm{min}}(\E_i))\\
&\leq&\frac{\sum\limits_{i=1}^m\rk(\E_i)-m}{p}\cdot\max\{0,~L_{\max}(\Omg^1_X)\}+\sum\limits_{i=1}^m\mathrm{I}(\E_i)
\end{eqnarray*}
\end{proof}

\begin{Theorem} \label{InstabDirIm}
If $\mu(\Omg^1_X)\geq 0$. Then for any torsion free sheaf $\E$ on $X$, we have
\begin{eqnarray*}
\mathrm{I}({F_X}_*(\E))&\leq&\Bigg\{\frac{\max\limits_{0\leq l\leq n(p-1)}\{\sum\limits_{q=0}^{l(p)}(-1)^q\cdot C_n^q\cdot C_{n+l-qp-1}^{l-qp}\}+\rk(\E)-2}{p}\cdot\max\{0,~L_{\max}(\Omg^1_X)\}\\
& &+\mathrm{min}\{~l,~[\frac{n}{2}](p-1)~\}\cdot(\frac{n-1}{p}
\cdot\max\{0,~L_{\max}(\Omg^1_X)\}+\mathrm{I}(\Omg^1_X))\\
& &+\mathrm{I}(\E)\Bigg\}\cdot p^{n-1}\cdot\rk(\E)
\end{eqnarray*}
where $l(p)\geq 0$ is the unique integer such that $0\leq l-l(p)<p$.
\end{Theorem}

\begin{proof} For any integer $0\leq l\leq n(p-1)$, by \cite[Lemma 4.3]{Sun08}, Theorem \ref{InstabTl} and Proposition \ref{Tensor}, we have
\begin{eqnarray*}
\mathrm{I}(\E\otimes_{\mathscr{O}_X}\mathrm{T}^l(\Omg^1_X))&\leq&
\frac{\rk(\mathrm{T}^l(\Omg^1_X))+\rk(\E)-2}{p}\cdot\max\{0,~L_{\max}(\Omg^1_X)\}\\
& &+\mathrm{I}(\mathrm{T}^l(\Omg^1_X))+\mathrm{I}(\E)\\
&\leq&\frac{\sum\limits_{q=0}^{l(p)}(-1)^q\cdot C_n^q\cdot C_{n+l-qp-1}^{l-qp}+\rk(\E)-2}{p}\cdot\max\{0,~L_{\max}(\Omg^1_X)\}\\
& &+\mathrm{min}\{~l,~[\frac{n}{2}](p-1)~\}\cdot(\frac{n-1}{p}
\cdot\max\{0,~L_{\max}(\Omg^1_X)\}+\mathrm{I}(\Omg^1_X))\\
& &+\mathrm{I}(\E)\\
& &\text{where}~l(p)\geq 0~\text{is the unique integer such that}~0\leq l-l(p)<p.
\end{eqnarray*}
where the first inequality is followed by Proposition \ref{Tensor}, and the second inequality is followed by Theorem \ref{InstabTl}.
Submit the above inequality into inequality (\ref{InsDIFrob}), we get the upper bound of $\mathrm{I}({F_X}_*(\E))$ as expected.
\end{proof}

We give some sufficient conditions for slope semi-stability of Frobenius direct images of torsion free sheaves as following.
\begin{Proposition}\label{FroDirIm}
If $\Omg^1_X$ is a slope strongly semi-stable sheaf with $\mu(\Omg^1_X)\geq 0$. Then ${F_X}_*(\E)$ is slope semi-stable whenever $\E$ is slope strongly semi-stable. Moreover, if $\Omg^1_X$ is a slope semi-stable sheaf with $\mu(\Omg^1_X)=0$, then ${F_X}_*(\E)$ is slope strongly semi-stable whenever $\E$ is slope semi-stable.
\end{Proposition}

\begin{proof} By Proposition \ref{SemiStabT^l}, we have $\mathrm{T}^l(\Omg^1_X)(0\leq l\leq\rk(\E)(p-1))$ are slope strongly semi-stable. Then $\E\otimes_{\mathscr{O}_X}\mathrm{T}^l(\Omg^1_X)(0\leq l\leq n(p-1))$ are also slope strongly semi-stable, since tensor product of slope strongly semi-stable sheaves is also slope strongly semi-stable. This implies the slope semi-stability of ${F_X}_*(\E)$ by \cite[Corollary 4.9]{Sun08}. Moreover, if $\mu(\Omg^1_X)=0$, then slope semi-stability of sheaves is equivalent to the slope strongly semi-stability. Hence ${F_X}_*(\E)$ is slope strongly semi-stable whenever $\E$ is slope semi-stable, which is also an immediate corollary of \cite[Theorem 2.1]{MehtaRamanathan83} and Theorem \ref{InstabDirIm}.
\end{proof}

\section{Slope (semi)-stability of sheaves of locally exact and closed differential forms}

Let $$\Omg^\bullet_X:0\longrightarrow\mathscr{O}_X\stackrel{d}{\longrightarrow}\Omg^1_X\stackrel{d_1}{\longrightarrow}
\Omg^2_X\stackrel{d_2}{\longrightarrow}\cdots\stackrel{d_{n-1}}{\longrightarrow}\Omg^n_X\longrightarrow 0$$
be the \emph{de Rham complex} of $X$. Taking Frobenius push forwards ${F_X}_*$ to $\Omg^\bullet_X$, we obtain the following complex ${F_X}_*(\Omg^\bullet_{X/S})$ on $X$:\\
\centerline{$0\longrightarrow{F_X}_*(\mathscr{O}_X)\stackrel{{F_X}_*(d)}{\longrightarrow}{F_X}_*(\Omg^1_X)\stackrel{{F_X}_*(d_1)}{\longrightarrow}
{F_X}_*(\Omg^2_X)\stackrel{{F_X}_*(d_2)}{\longrightarrow}$}
\centerline{$\cdots\stackrel{{F_X}_*(d_{n-2})}{\longrightarrow}{F_X}_*(\Omg^{n-1}_X)\stackrel{{F_X}_*(d_{n-1})}{\longrightarrow}{F_X}_*(\Omg^n_X)\longrightarrow 0.$}

The subsheaves $(0\leq i\leq n)$
$$B^i_X:=\mathrm{Im}({F_X}_*(d_{i-1}):{F_X}_*(\Omg^{i-1}_X)\longrightarrow {F_X}_*(\Omg^i_X))$$
$$Z^i_X:=\Ker({F_X}_*(d_i):{F_X}_*(\Omg^i_X)\longrightarrow {F_X}_*(\Omg^{i+1}_X))$$
of ${F_X}_*(\Omg^i_X)$ are called the \emph{sheaf of locally exact $i$-forms} and \emph{sheaf of locally closed $i$-forms}. For the sake of convenience, we set $B^0_X:=0$, $Z^0_X:=\mathscr{O}_X$, $Z^n_X:={F_X}_*(\omega_X)$, where $\omega_X:=\Omg^n_X$ is the canonical invertible sheaf of $X$. By Cartier isomorphism \cite[Theorem 7.2]{Katz70}, we have isomorphisms
$\Omg^i_X\cong Z^i_X/B^i_X(0\leq i\leq n)$.

For a smooth projective curve $X$ of genus $g\geq 2$. M. Raynaud \cite{Raynaud82} showed that $B^1_X$ is slope semi-stable and K. Joshi \cite{Joshi04} proved that $B^1_X$ is indeed slope stable. For a smooth projective surface $X$ such that $\Omg^1_X$ is slope semi-stable with $\mu(\Omg^1_X)>0$, Y. Kitadai and H. Sumihiro \cite{KitadaiSumihiroII08} showed that $B^1_X$ and $B^2_X$ are also slope semi-stable. Moreover, X. Sun \cite{Sun10ii} showed that $B^1_X$ and $B^2_X$ are indeed slope stable, $Z^1_X$ is slope semi-stable if $\mathrm{char}(k)=3$ and $Z^1_X$ is slope stable if $\mathrm{char}(k)>3$. In the higher dimensional case, X. Sun \cite[Theorem 2.3]{Sun10ii} showed that when $\mathrm{T}^l(\Omg^1_X)(0\leq l<n(p-1))$ are slope semi-stable with $\mu(\Omg^1_X)>0$, then $B^1_X$ is slope stable.

Fix an ample divisor $H$ on $X$, then for any torsion free sheaf $\E$ on $X$, we have the following formula (cf. \cite[Lemma 4.2]{Sun08})
$$\mu({F_X}_*(\E))=\frac{1}{p}\mu(F^*_X{F_X}_*(\E))=\frac{n\cdot(p-1)}{2p}\cdot\mu(\Omg^1_X)+\frac{\mu(\E)}{p}.$$
Using induction on $m$, it is easy to induce the following formula
$$\mu({F^m_X}_*(\E))=\frac{n\cdot(p^m-1)}{2p^m}\cdot\mu(\Omg^1_X)+\frac{\mu(\E)}{p^m}.$$
It follows that
$$\dg({F_X}_*(\Omg^i_X))=\frac{n\cdot C^i_np^{n-1}(p-1)}{2}\cdot\mu(\Omg^1_X)+i\cdot C^i_np^{n-1}\cdot\mu(\Omg^1_X).$$

\begin{Lemma} \label{Lemma:BxZx}
For any integer $0\leq i\leq n$, we have
$$\rk(B^i_X)=C^{i-1}_{n-1}(p^n-1),~\rk(Z^i_X)=C^{i-1}_{n-1}(p^n-1)+C^i_n$$
$$\dg(B^i_X)=\frac{n\cdot C^{i-1}_{n-1}p^{n-1}(p-1)}{2}\cdot\mu(\Omg^1_X)+\sum\limits_{j=1}^{i-1}(-1)^{i+j+1}j\cdot C^j_n(p^{n-1}-1)\cdot\mu(\Omg^1_X)$$
$$\dg(Z^i_X)=\frac{n\cdot C^{i-1}_{n-1}p^{n-1}(p-1)}{2}\cdot\mu(\Omg^1_X)+\sum\limits_{j=1}^{i-1}(-1)^{i+j+1}j\cdot C^j_n(p^{n-1}-1)\cdot\mu(\Omg^1_X)+i\cdot C^i_n\cdot\mu(\Omg^1_X).$$
\end{Lemma}

\begin{proof}
Using induction on $i$. When $i=0$, there is nothing to prove. Suppose the Lemma is true for $i-1$.
Consider the exact sequence
$$0\rightarrow Z^{i-1}_X\rightarrow{F_X}_*(\Omg^{i-1}_X)\rightarrow B^i_X\rightarrow 0.$$ Then we have $\rk(B^i_X)=\rk({F_X}_*\Omg^{i-1}_X)-\rk(Z^{i-1}_X)=C^{i-1}_{n-1}(p^n-1)$.
and
$$\dg(B^i_X)=\frac{n\cdot C^{i-1}_{n-1}p^{n-1}(p-1)}{2}\cdot\mu(\Omg^1_X)+\sum\limits_{j=1}^{i-1}(-1)^{i+j+1}j\cdot C^j_n(p^{n-1}-1)\cdot\mu(\Omg^1_X).$$

On the other hand, by exact sequence $0\rightarrow B^i_X\rightarrow Z^i_X\rightarrow \Omg^i_X\rightarrow 0$,
we have $\rk(Z^i_X)=\rk(B^i_X)+\rk(\Omg^i_X)=C^{i-1}_{n-1}(p^n-1)+C^i_n$, and
$$\dg(Z^i_X)=\frac{n\cdot C^{i-1}_{n-1}p^{n-1}(p-1)}{2}\cdot\mu(\Omg^1_X)+\sum\limits_{j=1}^{i-1}(-1)^{i+j+1}j\cdot C^j_n(p^{n-1}-1)\cdot\mu(\Omg^1_X)+i\cdot C^i_n\cdot\mu(\Omg^1_X).$$
This completes the proof of the Lemma.
\end{proof}

\begin{Proposition}\label{Prop:InstZiX}
If $\mu(\Omg^1_X)>0$. Then
\begin{itemize}
    \item[$(1)$.] $Z^i_X(1\leq i<\frac{n}{2})$ are never slope semi-stable.
    \item[$(2)$.] If $n\geq 3$ and $\mathrm{T}^l(\Omg^1_X)(0\leq l<n(p-1))$ are slope semi-stable. Then the Harder-Narasimhan filtration of $Z^1_X$ is $$0\subset B^1_X\subset Z^1_X.$$
\end{itemize}
\end{Proposition}

\begin{proof}
$(1)$. Consider the exact sequence $0\rightarrow B^i_X\rightarrow Z^i_X\rightarrow \Omg^i_X\rightarrow 0$.
To prove $Z^i_X$ is not slope semi-stable, it is enough to prove $\mu(B^i_X)>\mu(\Omg^i_X)$, i.e.
$$\frac{n\cdot C^{i-1}_{n-1}p^{n-1}(p-1)+2\sum\limits_{j=1}^{i-1}(-1)^{i+j+1}jC^j_n(p^{n-1}-1)}{2C^{i-1}_{n-1}(p^n-1)}>i.$$
Since $\sum\limits_{j=1}^{i-1}(-1)^{i+j+1}j\cdot C^j_n(p^{n-1}-1)>0$, so we only have to show that
$$\frac{n\cdot C^{i-1}_{n-1}p^{n-1}(p-1)}{2C^{i-1}_{n-1}(p^n-1)}=\frac{n(p^n-p)}{2(p^n-1)}>i.$$
The above inequality is trivial when $n>2\cdot i\geq 2$.

$(2)$. Since $B^1_X$ is slope semi-stable(cf. \cite[Theorem 2.3]{Sun10ii}), and by Lemma \ref{Lemma:BxZx} we have $\mu(B^1_X)>\mu(\Omg^1_X)$ if $n\geq 3$. Then by exact sequence $$0\rightarrow B^1_X\rightarrow Z^1_X\rightarrow \Omg^1_X\rightarrow 0,$$ we know the Harder-Narasimhan filtration of $Z^1_X$ is $0\subset B^1_X\subset Z^1_X$.
\end{proof}

\begin{Proposition}\label{Prop:BxZx0}
If $\Omg^1_X$ is slope semi-stable with $\mu(\Omg^1_X)=0$. Then $B^i_X(1\leq i\leq n)$ and $Z^i_X(1\leq i\leq n-1)$ are slope strongly semi-stable.
\end{Proposition}

\begin{proof}
By \cite[Theorem 2.1]{MehtaRamanathan83}, we have $\Omg^1_X$ is slope strongly semi-stable, hence $\Omg^{i}_X$ is slope strongly semi-stable for any integer $1\leq i\leq n$. Then by Proposition \ref{FroDirIm}, ${F_X}_*(\Omg^{i}_X)(1\leq i\leq n)$ are slope strongly semi-stable with $\mu({F_X}_*(\Omg^{i}_X))=0$. By Lemma \ref{Lemma:BxZx} we have $\mu(B^i_X)=\mu(Z^i_X)=0$, $1\leq i\leq n$. Then the Proposition follows from the exact sequence $0\rightarrow Z^{i-1}_X\rightarrow{F_X}_*(\Omg^{i-1}_X)\rightarrow B^i_X\rightarrow 0.$
\end{proof}

\begin{Proposition}\label{Prop:BnX}
If $\mu(\Omg^1_X)>0$ and $\mathrm{T}^l(\Omg^1_X)(0\leq l\leq n(p-1))$ are slope semi stable. Then for any subsheaf $B\subset {F_X}_*\omega_X$ with $0<\rk(B)<\rk({F_X}_*\omega_X)$, we have
$$\mu(B)-\mu(B^n_X)\leq -\frac{n(p-1)(p^n-\rk(B)-1)}{2p(p^n-1)\cdot \rk(B)}\mu(\Omg^1_X).$$
In particular, $B^n_X$ is slope stable.
\end{Proposition}

\begin{proof} Since $0<\rk(B)<\rk({F_X}_*\omega_X)$, by Theorem \cite[Theorem 2.3]{Sun10ii}, we have
$$\mu(B)\leq\mu({F_X}_*\omega_X)-\frac{n(p-1)}{2p\cdot\rk(B)}\cdot\mu(\Omg^1_X)=\frac{n(p+1)\cdot \rk(B)-n(p-1)}{2p\cdot \rk(B)}\cdot\mu(\Omg^1_X).$$
By Lemma \ref{Lemma:BxZx} and the combinational formula $\sum\limits_{j=0}\limits^n(-1)^jj\cdot C^j_n=0$, we have
$$\mu(B)-\mu(B^n_X)\leq-\frac{n(p-1)(p^n-\rk(B)-1)}{2p(p^n-1)\cdot \rk(B)}\cdot\mu(\Omg^1_X).$$

Let sub-sheaf $B\subseteq B^n_X$ with $0<\rk(B)<\rk(B^n_X)=p^n-1$. Then $\mu(B)<\mu(B^n_X)$. Hence, $B^n_X$ is slope stable.
\end{proof}

\end{document}